\documentclass{elsarticle}

\usepackage{graphicx}
\usepackage{hyperref}
\usepackage{amsfonts}
\usepackage{epstopdf}
\usepackage{algorithmic}
\usepackage{amssymb,amsmath,epsfig,multirow}
\usepackage[displaymath, mathlines,running]{lineno}
\usepackage{mathtools} % Bonus
\usepackage{color}
\usepackage{lineno,hyperref}
\usepackage{rotating}
\usepackage{makeidx}
\usepackage{lscape}
\usepackage{float}
\usepackage{epsfig}
\usepackage{amsmath,amssymb}
\usepackage{enumitem}
\usepackage{epic,eepic}
\usepackage{overpic}
\usepackage{color}
\usepackage{amsfonts}
\usepackage{graphicx}
\usepackage{enumitem}
\usepackage{pgf,tikz,pgfplots}
\pgfplotsset{compat=1.15}
\usepackage{mathrsfs}
\usetikzlibrary{arrows}

\newtheorem{theorem}{Theorem}[section]

\newtheorem{proposition}[theorem]{Proposition}
\newtheorem{remark}{Remark}
\newenvironment{proof}{\smallskip\noindent{{\it Proof.}}\hskip \labelsep}%
           {\hfill\penalty10000\raisebox{-.09em}{\large\bf\rm $\blacksquare$}\par\medskip}

\newtheorem{lemma}[theorem]{Lemma}

\begin{document}
%\footnotemark[\value{$\bigstar$}]
\begin{frontmatter}

\title{Numerical integration rules based on B-spline bases}\tnotetext[label1]{This research has been supported by project CIAICO/2021/227 and by grant PID2020-117211GB-I00 funded by MCIN/AEI/10.13039/501100011033}

\author[UV]{Dionisio F. Y\'a\~nez}
\ead{dionisio.yanez@uv.es}
\date{Received: date / Accepted: date}

\address[UV]{Departamento de Matem\'aticas, Facultad de Matemáticas. Universidad de Valencia. Valencia, Spain.}
\vspace{-0.5cm}

\begin{abstract}
In this work, we present some new integration formulas for any order of accuracy as an application of the B-spline relations obtained in \cite{amatetalexplicit}. The resulting rules are defined as a perturbation of the trapezoidal integration method. We prove the order of approximation and extend the results to several dimensions. Finally, some numerical experiments are performed in order to check the theoretical results.
\end{abstract}


\end{frontmatter}

\section{Introduction}
Integration rules are a powerful tool to approximate the value of an integral of a function at an interval, evaluating it on a grid on this interval. There exist some very well-known formulas as Newton-Cotes where the key is to approximate the function using Lagrange interpolation. Therefore, we consider the following classical problem, to approximate
 \begin{equation}\label{problema0}
\int_a^b f(x)dx.
\end{equation}
For example, if we divide our interval using $N+1$ equidistance points, define $x_i=a+ih_N$, $i=0,\hdots,N$, $h_N:=\frac{b-a}{N}$, and approximate the function with a polynomial of degree one to calculate the integral in each interval $[x_{i},x_{i+1}]$, $i=0,\hdots,N-1$, we get the classical well-known formula called trapezoidal integration rule:
\begin{equation}\label{trapezoide}
T_{[a,b],N}(f):= \frac{h_N}{2}(f(x_{0})+f(x_{N})) +h_N\sum_{i=1}^{N-1}f(x_{i}),
\end{equation}
Recently, new formulas to obtain an approximation of a function from the cell-average data have been developed, \cite{amatetalexplicit}. In particular, they offer a relation between approximations based on point-value data and cell-average data using B-spline functions. In the present paper we use these results to give some new flexible integration formulas for any order of accuracy $p$, defining them as a correction of $T_{[a,b],N}$. We check that these formulas are a particular case of Euler-McClaurin rules \cite{devries} when the values of the derivatives on the extreme of interval are approximated using Taylor expansions. We structure the paper as follows: Firstly, we show the general formula and prove the order of accuracy for one dimension, these rules can be extended for $k$ dimensions just using a tensor-product strategy and finally, we present some numerical experiments and conclusions.

\section{Formulation of the problem and review}
Using the notation employed by Amat et al. in \cite{amatetalexplicit}, we suppose $f\in\mathcal{C}^{p+1}(\mathbb{R})$,
$h>0$ a constant, $B_p(x)$ the $p$-degree B-spline supported on
$I_p =\left[-\frac{p+1}{2},\frac{p+1}{2}\right]$, with equidistant knots
$S_p = \left\{-\frac{p+1}{2},\hdots,\frac{p+1}{2}\right\},$ and define the vector
$f_{n,p}=(f_{n-\left\lfloor \frac{p}{2} \right\rfloor},\hdots,f_{n+\left\lfloor \frac{p}{2} \right\rfloor}),$
with $f_i=f(ih)$ and where $\left\lfloor \cdot \right\rfloor$ is the floor function. It is well-known that the following quasi-interpolation operator approximates $f$ and it is exact for polynomials of degree up $p$, see \cite{speleers}:
\begin{equation}\label{operatorQ}
Q_p(f)(x)=\sum_{n\in \mathbb{Z}}L_{p}(f_{n,p})B_p\left(\frac{x}{h}-n \right),
\end{equation}
where $L_p:\mathbb{R}^{2\left\lfloor \frac{p}{2} \right\rfloor+1}\to \mathbb{R}$ is the linear function defined as:
\begin{equation}\label{operadorL}
L_p(f_{n,p})=\sum_{j=-\left\lfloor \frac{p}{2} \right\rfloor}^{\left\lfloor \frac{p}{2} \right\rfloor} c_{p,j} f_{n+j},
\end{equation}
with the coefficients $c_{p,j}$, $j=-\left\lfloor \frac{p}{2} \right\rfloor,\hdots,\left\lfloor \frac{p}{2} \right\rfloor$ showed in Table \ref{tablaCs} for $j=0,1,2$. A general expression for these coefficients can be found in \cite{speleers}. These values are calculated to guarantee the reproduction of polynomials of the operator $Q_p$.
\begin{table}[htbp]
  \begin{center}
   \begin{tabular}{crrrrrrr}\hline
                             & $c_{1,j}$ & $c_{2,j}$ & $c_{3,j}$ & $c_{4,j}$& $c_{5,j}$ \\    \hline%& $c_{6,j}$  &$c_{7,j}$     \\    \hline
                 $ j=0$      &  $1$  & $5/4 $    & $4/3$    & $319/192$  & $73/40$ \\ %& $79879/34560$  & $2452/945$             \\
                 $ j=1$      &       & $-1/8 $   & $-1/6$   & $-107/288$ & $-7/15$ \\ %& $-37003/46080$ & $-1657/1680$             \\
                 $ j=2$      &       &           &          & $47/1152$  & $13/240$ \\\hline %& $751/4608$     & $22/105$     \\
%                 $ j=3$      &       &           &          &            &          & $-2159/138240$ & $-311/15120$\\\hline
%$||L_p||$                    & $1$   & $3/2$     & $5/3$    & $179/72$   & $43/15$  & $9233/2160$    & $4751/945$ \\\hline
   \end{tabular}
    \caption{Values $c_{p,k}$. Note that $c_{p,j}=c_{p,-j}$, $j=0,\hdots,\left\lfloor\frac{p}{2}\right\rfloor$}
    \label{tablaCs}
  \end{center}
\end{table}

With these ingredients, we introduce the problem. Let $a,b$ be two real constants and $f:[a,b]\to \mathbb{R}$ an integrable function, and denote the measure of the interval as $h_1=b-a$, we want to approximate:
\begin{equation}\label{problema1}
\int_a^b f(x)dx =\int_0^{b-a} f(x+a) dx =\int_0^{b-a} f^{a}(x) dx=\int_0^{h_1} f^{a}(x) dx=\int_{\frac {h_1}{2} -\frac{h_1}{2}}^{\frac {h_1}{2} +\frac{h_1}{2}} f^{a}(x) dx= h_1(f^a\ast\omega_{h_1})\left(\frac{h_1}{2}\right)=h_1\bar{f}^a\left(\frac{h_1}{2}\right)
\end{equation}
where $f^a:[0,b-a]\to \mathbb{R}$ defined as $f^a(x)=f(x+a)$. We also define the cell-average of $f$ as $\bar{f}^a=f^a\ast\omega_{h_1}$ where
\begin{equation*}
\omega_{h_1}(x)= \frac{1}{h_1}\chi_{[-\frac{1}{2},\frac{1}{2}]}\left(\frac{x}{h_1}\right)=\left\{
         \begin{array}{ll}
           h_1^{-1}, & \hbox{if}\,\, x\in [-\frac{h_1}{2},\frac{h_1}{2}], \\
           0, & \hbox{otherwise}.\\
         \end{array}
       \right.
\end{equation*}
A relation between the approximations using point-value and cell-average discretizations proved in \cite{amatetalexplicit} can be summarize with the following equality for polynomials of degree up to $p$:
If $P\in\Pi_{p}(\mathbb{R})$ then
\begin{equation}\label{equality}
\sum_{n\in\mathbb{Z}}L_{p+1}(\bar P_{n-\left\lfloor \frac{p+1}{2} \right\rfloor},\hdots,\bar{P}_{n+\left\lfloor \frac{p+1}{2} \right\rfloor})B_{p+1}\left(\frac{x}{h}-n\right)=\sum_{n\in\mathbb{Z}}L_{p}(P_{n-\left\lfloor \frac{p}{2} \right\rfloor},\hdots,P_{n+\left\lfloor \frac{p}{2} \right\rfloor})B_{p+1}\left(\frac{x}{h}-n\right).
\end{equation}
Note that $\bar f_i$ is the cell-average value at the interval $[ih-h/2,ih+h/2]$, i.e. $\bar{f}_{i}=\bar{f}(ih)=\frac{1}{h}\int_{ih-\frac{h}{2}}^{ih+\frac{h}{2}}f(x)dx$. Also, it is easy to check the following result.

\begin{theorem}\label{prop1}
Let $p$ be an even number, $L_{p}$ the linear function defined in Eq. \eqref{operadorL} and $P\in\Pi_{p+1}(\mathbb{R})$ then:
$$L_{p+1}\left(\bar P_{n-\frac{p}{2}},\hdots,\bar{P}_{n+\frac{p}{2}}\right)=L_{p}\left(P_{n-\frac{p}{2}},\hdots,{P}_{n+\frac{p}{2}}\right).$$
\end{theorem}

\subsection{Integration rule for one subinterval}
We start with the approximation of the integral over one interval. We suppose $P\in\Pi_{p}(\mathbb{R})$, then by \eqref{equality} we have
\begin{equation}\label{polynomial}
\bar{P}^a(x)=\sum_{n\in\mathbb{Z}}L_{p+1}(\bar{P}^a_{n-\left\lfloor \frac{p+1}{2} \right\rfloor},\hdots,\bar{P}^a_{n+\left\lfloor \frac{p+1}{2} \right\rfloor})B_{p+1}\left(\frac{x}{h_1}-n\right)=\sum_{n\in\mathbb{Z}}L_{p}(P^a_{n-\left\lfloor \frac{p}{2} \right\rfloor},\hdots,P^a_{n+\left\lfloor \frac{p}{2} \right\rfloor})B_{p+1}\left(\frac{x}{h_1}-n\right),
\end{equation}
then, we can define for any $f$ integrable function:
\begin{equation*}
\bar{f}^a\left(\frac{h_1}2\right)\approx \sum_{n\in\mathbb{Z}}L_{p}(f^a_{n-\left\lfloor \frac{p}{2} \right\rfloor},\hdots,f^a_{n+\left\lfloor \frac{p}{2} \right\rfloor})B_{p+1}\left(\frac{1}{2}-n\right).
\end{equation*}
From the support of $B_{p}$ is $I_p=[-\frac{p+1}{2},\frac{p+1}{2}]$ and for $p\geq 1$, we have
$B_{p+1}\left(\frac{p+2}{2}\right)=B_{p+1}\left(-\frac{p+2}{2}\right)=0,$ we can determine the values $n$ where $B_{p+1}\left(\frac{1}{2}-n\right)\neq 0$, thus
$ -\frac{p+2}{2}<\frac{1}{2}-n<\frac{p+2}{2}$ then $-\frac{p+1}{2} <n <\frac{p+1}{2}+1.$
Therefore, our approximation will be:
\begin{equation}\label{aprox}
\begin{split}
\int_a^b f(x)dx&=h_1\bar{f}^a\left(\frac{h_1}2\right)\approx h_1\sum_{n\in\mathbb{Z}}L_{p}(f^a_{n-\left\lfloor \frac{p}{2} \right\rfloor},\hdots,f^a_{n+\left\lfloor \frac{p}{2} \right\rfloor})B_{p+1}\left(\frac{1}{2}-n\right)\\
&=h_1\sum_{n=-\left\lfloor \frac{p}{2} \right\rfloor}^{\left\lfloor \frac{p}{2} \right\rfloor+1}L_{p}(f^a_{n-\left\lfloor \frac{p}{2} \right\rfloor},\hdots,f^a_{n+\left\lfloor \frac{p}{2} \right\rfloor})B_{p+1}\left(\frac{1}{2}-n\right)\\&=h_1\sum_{j=-2\left\lfloor \frac{p}{2} \right\rfloor}^{0}\left(\sum_{r=j-1-\left\lfloor \frac{p}{2} \right\rfloor}^{j+\left\lfloor \frac{p}{2} \right\rfloor}c_{p,r}B_{p+1}\left(r-j+\frac{1}{2}\right)\right)(f^a_{j}+f^a_{1-j})=h_1\sum_{j=-2\left\lfloor \frac{p}{2} \right\rfloor}^{2\left\lfloor \frac{p}{2} \right\rfloor+1}\tau_{p,j}f^a_{j}\\
&=h_1\sum_{j=-2\left\lfloor \frac{p}{2} \right\rfloor}^{2\left\lfloor \frac{p}{2} \right\rfloor+1}\tau_{p,j}f(x_{j}),\\
%
%\sum_{-\frac{p+1}{2}<n<\frac{p+1}{2}+1}L_{p}(f^a_{n-\left\lfloor \frac{p}{2} \right\rfloor},\hdots,f^a_{n+\left\lfloor \frac{p}{2} \right\rfloor})B_{p+1}\left(\frac{1}{2}-n\right)\\
%&=\sum_{n=-k}^{k+1}L_{p}(f^a_{n-k},\hdots,f^a_{n+k})B_{p+1}\left(\frac{1}{2}-n\right)\\
%&=\sum_{j=0}^{p-1}\left(\sum_{r=0}^{j}c_{p,r-k}B_{p+1}\left(r-j+k+\frac{1}{2}\right)\right)(f^a_{p-j}+f^a_{j+1-p})\\
\end{split}
\end{equation}
where we call B-integration terms to
\begin{equation}\label{integrationterms}
\tau_{p,j}=\sum_{r=j-1-\left\lfloor \frac{p}{2} \right\rfloor}^{j+\left\lfloor \frac{p}{2} \right\rfloor} c_{p,r}B_{p+1}\left(r-j+\frac{1}{2}\right),
\end{equation}
in Table \ref{tablaTs} we show some values. Note that if $p=1$ the trapezoidal simple rule is recovered.
\begin{table}[htbp]
  \begin{center}
   \begin{tabular}{crrrrr}\hline
                             & $\tau_{1,j}$ & $\tau_{2,j}$ & $\tau_{3,j}$ & $\tau_{4,j}$& $\tau_{5,j}$       \\    \hline
                 $ j=0$      &  $1/2$   & $103/192$    & $19/36$     & $5.446148907696758e-01$   &  $5.371643518518517e-01$   \\
                 $ j=-1$      &         & $-13/384 $   & $-1/48$     & $-4.241988570601853e-02$   & $-2.918981481481481e-02$    \\
                 $ j=-2$      &         & $-1/384 $    & $-1/144 $   & $-4.626916956018520e-03$  & $-1.168981481481482e-02$   \\
                 $ j=-3$      &         &              &             & $2.421287254050926e-03$  & $3.640046296296296e-03$           \\
                 $ j=-4$      &         &              &             & $1.062463831018518e-05$   & $7.523148148148149e-05$           \\\hline
%$||L_p||$                    & $1$     & $3/2$     & $5/3$    & $179/72$   & $43/15$  & $9233/2160$    & $4751/945$ \\\hline
   \end{tabular}
    \caption{Values $\tau_{p,j}$, $j=0,\hdots,-4$. Note that $\tau_{p,j}=\tau_{p,1-j}.$}
    \label{tablaTs}
  \end{center}
\end{table}
\begin{lemma}\label{lema1}
Let $p$ be a natural number, and the B-integration terms $\tau_{p,j}$, $j=-2\left\lfloor \frac{p}{2} \right\rfloor,\hdots,2\left\lfloor \frac{p}{2} \right\rfloor+1$ defined in Eq. \eqref{integrationterms}, with $c_{p,j}$, showed in  Table \eqref{tablaCs}, then $\tau_{p,j}=\tau_{p,1-j}$, $\sum_{j=-2\left\lfloor \frac{p}{2} \right\rfloor}^{2\left\lfloor \frac{p}{2} \right\rfloor+1}\tau_{p,j}=1,$ and $\sum_{j=-2\left\lfloor \frac{p}{2} \right\rfloor}^{0}\tau_{p,j}=\frac{1}{2}.$
\end{lemma}
\begin{proof}
 Since $B_{p+1}(x)=B_{p+1}(-x)$, for all $x\in\mathbb{R}$ and $c_{p,-r}=c_{p,r}$ for all $r=-\left\lfloor \frac{p}{2} \right\rfloor,\hdots,\left\lfloor \frac{p}{2} \right\rfloor$ we get:
\begin{equation}
\begin{split}
\tau_{p,1-j}&=\sum_{r=-j-\left\lfloor \frac{p}{2} \right\rfloor}^{-j+1+\left\lfloor \frac{p}{2} \right\rfloor} c_{p,r}B_{p+1}\left(r+j-\frac{1}{2}\right)=
\sum_{r=j-1-\left\lfloor \frac{p}{2} \right\rfloor}^{j+\left\lfloor \frac{p}{2} \right\rfloor} c_{p,-r}B_{p+1}\left(-r+j-\frac{1}{2}\right)=\sum_{r=j-1-\left\lfloor \frac{p}{2} \right\rfloor}^{j+\left\lfloor \frac{p}{2} \right\rfloor} c_{p,r}B_{p+1}\left(r-j+\frac{1}{2}\right)=\tau_{p,j}.
\end{split}
\end{equation}
 It follows that $\sum_{j=-2\left\lfloor \frac{p}{2} \right\rfloor}^{2\left\lfloor \frac{p}{2} \right\rfloor+1}\tau_{p,j}=1$ from taking a function $P\equiv 1$ in \eqref{polynomial} and employing \eqref{aprox}.
\end{proof}

\subsection{Composite rules}
In order to design the composed rules, we study the following case: we divide the interval $[0,b-a]$ using $N+1$ points, define
$h:=h_N=\frac{b-a}{N}$ and design a grid $ih$ with $0\leq i\leq N$, we approximate the integral:
\begin{equation}\label{problema2}
\int_{ih}^{ih+h} f^a(x)dx =\int_{(i+\frac{1}{2})h-\frac h2}^{(i+\frac{1}{2})h+\frac h2} f^a(x) dx =\int_{(i+\frac{1}{2})h-\frac h2}^{(i+\frac{1}{2})h+\frac h2} f^a(x) dx= h(f^a\ast\omega^1_h)\left(ih+\frac h 2\right)=h\bar{f}^a\left(ih+\frac h 2\right)
\end{equation}
using again Eq. \eqref{equality}, we get:
\begin{equation}\label{aproximacionbuena}
\begin{split}
h\bar{f}^a\left(ih+\frac h2\right)&\approx \sum_{n\in\mathbb{Z}}L_{p}(f^a_{n-\left\lfloor \frac{p}{2} \right\rfloor},\hdots,f^a_{n+\left\lfloor \frac{p}{2} \right\rfloor})B_{p+1}\left(\frac{1}{2}-(n-i)\right)=\sum_{n=i-\left\lfloor \frac{p}{2} \right\rfloor}^{i+\left\lfloor \frac{p}{2} \right\rfloor+1}L_{p}(f^a_{n-\left\lfloor \frac{p}{2} \right\rfloor},\hdots,f^a_{n+\left\lfloor \frac{p}{2} \right\rfloor})B_{p+1}\left(\frac{1}{2}-(n-i)\right)\\
&=\sum_{j=-2\left\lfloor \frac{p}{2} \right\rfloor}^{0}\left(\sum_{r=j-1-\left\lfloor \frac{p}{2} \right\rfloor}^{j+\left\lfloor \frac{p}{2} \right\rfloor}c_{p,r}B_{p+1}\left(r-j+\frac{1}{2}\right)\right)(f^a_{i+j}+f^a_{i-j+1})=\sum_{j=-2\left\lfloor \frac{p}{2} \right\rfloor}^{2\left\lfloor \frac{p}{2} \right\rfloor+1}\tau_{p,j}f^a_{i+j},\\
%&=\sum_{j=-2\left\lfloor \frac{p}{2} \right\rfloor}^{0}\tau_{p,j}(f^a_{i+j}+f^a_{i-j+1}),\\
\end{split}
\end{equation}
We define the sum of the B-integration terms as $\xi_{p,i}=\sum_{j=-2\left\lfloor \frac{p}{2} \right\rfloor}^{i}\tau_{p,j},$
and by Lemma \ref{lema1} $\xi_{p,0}=1/2$ for all $p\in\mathbb{N}$. In Table \ref{tablaXis} we show some values $\xi_{p,i}.$
\begin{table}[htbp]
  \begin{center}
   \begin{tabular}{lrrrrr}\hline
                             & $\xi_{1,i}$ & $\xi_{2,i}$ & $\xi_{3,i}$ & $\xi_{4,i}$& $\xi_{5,i}$       \\    \hline
%                 $ j=4$      &  $1$     & $1$           & $1$     & $9.999893753616899e-01$   &  $9.999247685185181e-01$   \\
%                 $ j=3$      &  $1$     & $1$           & $1$     & $9.975680881076390e-01$   &  $9.962847222222218e-01$   \\
%                 $ j=2$      &  $1$     & $385/384 $    & $145/144 $     & $1.002195005063657e+00$   &  $1.007974537037037e+00$   \\
%                 $ j=1$      &  $1$     & $199/192 $    & $37/36$     & $1.044614890769676e+00$   &  $1.037164351851851e+00$   \\
                 $ i=0$      &  $1/2$   & $1/2$         & $1/2$     & $1/2$                         &  $1/2$   \\
                 $ i=-1$      &         & $-7/192 $     & $-1/36$     & $-4.461489076967595e-02$   & $-3.716435185185185e-02$    \\
                 $ i=-2$      &         & $-1/384 $     & $-1/144 $   & $-2.195005063657410e-03$  & $-7.974537037037042e-03$   \\
                 $ i=-3$      &         &               &             & $2.431911892361110e-03$  & $3.715277777777778e-03$           \\
                 $ i=-4$      &         &               &             & $1.062463831018518e-05$   & $7.523148148148149e-05$           \\\hline
%$||L_p||$                    & $1$     & $3/2$     & $5/3$    & $179/72$   & $43/15$  & $9233/2160$    & $4751/945$ \\\hline
   \end{tabular}
    \caption{Values $\xi_{p,i}$, $i=-4,\hdots,0$. Note that $\xi_{p,i}=1-\xi_{p,-i}$.}
    \label{tablaXis}
  \end{center}
\end{table}

To calculate the composed rule we use \eqref{aproximacionbuena}. Therefore, denoting by $T_{[a,b],N}$ the trapezoidal composed rule, Eq. \eqref{trapezoide}, we get by Lemma \ref{lema1} that:
\begin{equation}\label{problema3}
\begin{split}
\int_a^b f(x)dx &=\int_0^{b-a} f(x+a) dx =\int_0^{Nh} f^{a}(x) dx=\sum_{i=0}^{N-1}\int_{ih}^{(i+1)h} f^{a}(x) dx\approx  h\sum_{i=0}^{N-1} \left(\sum_{j=-2\left\lfloor \frac{p}{2} \right\rfloor}^{2\left\lfloor \frac{p}{2} \right\rfloor+1}\tau_{p,j}f^a_{i+j}\right)\\
%&\approx h\sum_{i=-2\left\lfloor \frac{p}{2} \right\rfloor}^{2\left\lfloor \frac{p}{2} \right\rfloor}\xi_{p,i}(f^a_{i}+f^a_{N-i}) +h\sum_{i=2\left\lfloor \frac{p}{2} \right\rfloor+1}^{N-2\left\lfloor \frac{p}{2} \right\rfloor-1}f^a_{i}\\
&= h\sum_{i=-2\left\lfloor \frac{p}{2} \right\rfloor}^{-1}\xi_{p,i}(f^a_{i}+f^a_{N-i}) + h\sum_{i=1}^{2\left\lfloor \frac{p}{2} \right\rfloor}\xi_{p,i}(f^a_{i}+f^a_{N-i})+  \frac{h}{2}(f^a_{0}+f^a_{N})+h\sum_{i=2\left\lfloor \frac{p}{2} \right\rfloor+1}^{N-2\left\lfloor \frac{p}{2} \right\rfloor-1}f^a_{i}\\
&= h\sum_{i=-2\left\lfloor \frac{p}{2} \right\rfloor}^{-1}\xi_{p,i}(f^a_{i}+f^a_{N-i}) + h\sum_{i=-2\left\lfloor \frac{p}{2} \right\rfloor}^{-1}\xi_{p,-i}(f^a_{-i}+f^a_{N+i})+  \frac{h}{2}(f^a_{0}+f^a_{N})+h\sum_{i=2\left\lfloor \frac{p}{2} \right\rfloor+1}^{N-2\left\lfloor \frac{p}{2} \right\rfloor-1}f^a_{i}\\
&= h\sum_{i=-2\left\lfloor \frac{p}{2} \right\rfloor}^{-1}\xi_{p,i}(f^a_{i}+f^a_{N-i}-f^a_{-i}-f^a_{N+i}) + \frac{h}{2}(f^a_{0}+f^a_{N})+h\sum_{i=1}^{N-1}f^a_{i}\\
&=T_{[a,b],N}(f) +h\sum_{i=1}^{2\left\lfloor \frac{p}{2} \right\rfloor}\xi_{p,-i}(f(x_{-i})-f(x_{i})+f(x_{N+i})-f(x_{N-i}))=:T_{[a,b],N}^p(f).
\end{split}
\end{equation}
We introduce some examples. When $p=1$, since $\xi_{1,0}=\frac{1}{2}$ the trapezoidal classical rule is recovered, i.e.
\begin{equation*}
\begin{split}
T_{[0,b-a],N}^1(f^a)=h\xi_{1,0}(f^a_{0}+f^a_{N}) +h\sum_{i=1}^{N-1}f^a_{i} = \frac{h}{2}(f^a_{0}+f^a_{N}) +h\sum_{i=1}^{N-1}f^a_{i}=T_{[0,b-a],N}(f^a)=T_{[a,b],N}(f).
\end{split}
\end{equation*}
In the case that $p=2$, we obtain the following rule:
\begin{equation*}
\begin{split}
T_{[a,b],N}^2(f)=&T_{[a,b],N}(f)+ h\sum_{i=1}^{2}\xi_{2,-i}(f^a_{-i}-f^a_{i}+f^a_{N+i}-f^a_{N-i})  \\
=& T_{[a,b],N}(f)-\frac{h}{384}(f(x_{-2})-f(x_{2})+f(x_{N+2})-f(x_{N-2}))-\frac{7h}{192}(f(x_{-1})-f(x_{1})+f(x_{N+1})-f(x_{N-1})),
\end{split}
\end{equation*}
and if $p=3$ we get:
\begin{equation*}
\begin{split}
T_{[a,b],N}^3(f)=& T_{[a,b],N}(f) + h\sum_{i=1}^{2}\xi_{3,-i}(f^a_{-i}-f^a_{i}+f^a_{N+i}-f^a_{N-i}) \\
=& T_{[a,b],N}(f)-\frac{h}{144}(f(x_{-2})-f(x_{2})+f(x_{N+2})-f(x_{N-2}))-\frac{h}{36}(f(x_{-1})-f(x_{1})+f(x_{N+1})-f(x_{N-1})).
\end{split}
\end{equation*}
Note that the number of function evaluations for each method concerning the trapezoidal method increases by $4\left\lfloor \frac{p}{2} \right\rfloor$. We prove the order of approximation in the following proposition.

\begin{proposition}
Let $p,N$ be natural numbers, $a,b\in\mathbb{R}$, $h=\frac{b-a}{N}$, $f\in \mathcal{C}^{p+1}(\Omega)$ if $p$ is odd and $f\in \mathcal{C}^{p+2}(\Omega)$ if $p$ is even with $[a,b]\subset [a-2\left\lfloor \frac{p}{2} \right\rfloor h,b+2\left\lfloor \frac{p}{2} \right\rfloor h] =\Omega$ and $T_{[a,b],N}^p(f)$ the integration rule defined in Eq. \eqref{problema3}, then
$$\left|\int_a^b f(x)dx - T_{[a,b],N}^p(f)\right|=\left\{
                                            \begin{array}{ll}
                                              O(h^{p+1}), & p \hbox{ is odd;} \\
                                              O(h^{p+2}), & p \hbox{ is even.}
                                            \end{array}
                                          \right.$$

\end{proposition}
The proof of this proposition is direct by Th. \ref{prop1} and Eq. \eqref{equality}. Finally, to extend our results to several dimensions we just use tensor-product strategy. This technique is analyzed and tested in \cite{amatetalexplicit}.

\begin{remark}
We can connect our formulas with the well-known rules developed by Euler-McClaurin. We introduce the general formula (see more details in \cite{devries}):
$$E_{[a,b],N}^{2p}(f)=T_{[a,b],N}(f)+h\sum_{k=1}^{p} \frac{B_{2k}}{(2k)!}(f^{2k-1)}(b)-f^{2k-1)}(a))+O(h^{2p+2}).$$
where $B_{l}$ are the Bernoulli numbers. It is not difficult to check that using an appropriate choice of the approximation to the derivative we recover our methods. Therefore, for example, if $p=1$, we approximate the derivative at $f'_j=f'(jh)$ using Taylor expansions with four points around $f_j$ obtaining:
\begin{equation}
f'_j=\alpha(f_{j-2}-f_{j+2})+\left(2\alpha+\frac{1}{2h}\right)(f_{j+1}-f_{j-1})+O(h^2),\\
\end{equation}
where $\alpha\in\mathbb{R}$ is a constant. We define a new family of integration rules as:
\begin{equation*}
\begin{split}
M_{\alpha,[a,b],N}^2(f):=&T_{[a,b],N}(f)+\frac{h^2}{12}\alpha(f^a_{-2}-f^a_{2}+f^a_{N+2}-f^a_{N-2})-\frac{h^2}{12}\left(2\alpha+\frac{1}{2h}\right)(f^a_{-1}-f^a_{1}+f^a_{N+1}-f^a_{N-1}).
\end{split}
\end{equation*}
Note that if $\alpha=-\frac{1}{32h}$ then $2\alpha+\frac{1}{2h}=\frac{7}{16h}$ and:
\begin{equation*}
\begin{split}
M_{-\frac{1}{32h},[a,b],N}^2\left(f\right)
&=T_{[a,b],N}(f)-\frac{h}{384}(f^a_{-2}-f_{2}+f^a_{N+2}-f^a_{N-2})-\frac{h^2}{12}\left(\frac{7}{16h}\right)(f^a_{-1}-f^a_{1}+f^a_{N+1}-f^a_{N-1})\\
&=T_{[a,b],N}(f)-\frac{h}{384}(f^a_{-2}-f_{2}+f^a_{N+2}-f^a_{N-2})-\frac{7h}{192}(f^a_{-1}-f^a_{1}+f^a_{N+1}-f^a_{N-1})=T^2_{[a,b],N}(f).
\end{split}
\end{equation*}
If $\alpha=-\frac{1}{12h}$ then $2\alpha+\frac{1}{2h}=\frac{1}{3h}$ and
\begin{equation*}
\begin{split}
M_{-\frac{1}{12h},[a,b],N}^2\left(f\right)
&=T_{[a,b],N}(f)-\frac{h}{144}(f^a_{-2}-f^a_{2}+f^a_{N+2}-f^a_{N-2})-\frac{h}{36}(f^a_{-1}-f^a_{1}+f^a_{N+1}-f^a_{N-1})=T^3_{[a,b],N}(f).
\end{split}
\end{equation*}
\end{remark}

\section{Numerical examples and conclusions}
We divide the numerical experiments in two parts in order to analyze the numerical order obtained and the number of evaluations comparing them with fourth order Simpson method defined as ($N$ even number):
$$S_{[a,b],N}=\frac{h}{3}\left(f(x_0)+f(x_N)+2\sum_{i=1}^{\frac{N}{2}-1}f(x_{2i})+4\sum_{i=1}^{\frac{N}{2}}f(x_{2i-1})\right). $$
We define the error and the numerical order for each integral of the function $f$ at the interval $[a,b]$ as:
 $$e^p_N=\left|\int_a^b f(x)dx - T_{[a,b],N}^p(f)\right|, \quad e^S_N=\left|\int_a^b f(x)dx - S_{[a,b],N}(f)\right|, \quad o_N^{\Xi}=\log_2\left(\frac{e_{N}^{\Xi}}{e_{2N}^{\Xi}}\right),$$
where $T_{[a,b],N}^p(f)$ is the integral rule defined in Eq. \eqref{problema3} and $\Xi=p,S$.
Note that for using the new B-spline integration formulas the number of evaluations is $M=N+1+4\left\lfloor \frac{p}{2} \right\rfloor$.

\subsection{Analysis of the numerical order}
We perform a classical example to solve the problem \eqref{problema0} for $a=0$, $b=1$ and the function $f_1(x)=e^{x^2}$ using the rules \eqref{problema3} with different $p$ values for $N=$ 80, 160, 320. In this case, we take as real solution the value 1.4626517459071815. We can see in Table \ref{tablaerrores} that for each $p$ even number, the methods $T^p$ and $T^{p+1}$ has the same order of approximation but the error produces for $p$ (even) methods are lower than $p+1$ methods. We can observe that, in this example, the necessary extra points for the B-spline formulas produces a bigger error when we compare with classical Simpson rule. We analyze it with more detail in next subsection.
 \begin{table}[H]
\begin{center}
\begin{tabular}{cccccccccccccc}
\hline
$N$   &  &    $e_N^1$       & $o_N^1$& &   $e_N^2$     & $o_N^2$ &&         $e_N^3$  & $o_N^3$  &&         $e_N^4$  & $o_N^4$   \\  \cline{1-1} \cline{3-4} \cline{6-7}  \cline{9-10}  \cline{12-13}
%  10  &  &   $4.5229e-03 $ &           & &$   1.1307e-04$    &  -         &&$1.6214e-04$    &  -         &&$   7.2872e-06$    &  -           \\
%  20  &  &   $1.1321e-03 $ &   1.9982  & &$   6.9870e-06$    &    4.0163  &&$9.9654e-06$    &    4.024   &&$   1.0942e-07$    &      6.0574    \\
%  40  &  &   $2.8312e-04 $ &   1.9995  & &$   4.3546e-07$    &    4.0041  &&$6.2026e-07$    &    4.006   &&$   1.6929e-09$    &      6.0142    \\
  80  &  &   $7.0787e-05 $ &   1.9999  & &$   2.7197e-08$    &    4.0010  &&$3.8726e-08$    &    4.001   &&$   2.6387e-11$    &      6.0035    \\
  160 &  &   $1.7697e-05 $ &   2.0000  & &$   1.6995e-09$    &    4.0003  &&$2.4197e-09$    &    4.000   &&$   4.1167e-13$    &      6.0022    \\
  320 &  &   $4.4243e-06 $ &   2.0000  & &$   1.0622e-10$    &    4.0001  &&$1.5122e-10$    &    4.000   &&$   5.9952e-15$    &      6.1015    \\
%\cline{1-1} \cline{3-4} \cline{6-7}  \cline{9-10}  \cline{12-13}
\hline   %\cline{1-1} \cline{3-4} \cline{6-7}  \cline{9-10}       \cline{12-13}
$N$   &  &    $e_N^5$       & $o_N^5$& &   $e_N^6$     & $o_N^6$ &&         $e_N^7$  & $o_N^7$  &&         $e_{N}^S$  & $o_{N}^S$   \\  \cline{1-1} \cline{3-4} \cline{6-7}  \cline{9-10}  \cline{12-13}
%  10  &  &   $   1.0352e-05$ &           & &$   6.8340e-07$    &  -         &&$   9.4721e-07$    &  -          &&$     2.9654e-05$    & \\
%  20  &  &   $   1.5458e-07$ &   6.0654  & &$   2.4657e-09$    &  8.1146    &&$   3.3938e-09$    &     8.1246  &&$     1.8790e-06$    &    3.9802\\
%  40  &  &   $   2.3883e-09$ &   6.0162  & &$   9.4438e-12$    &  8.0284    &&$   1.2977e-11$    &     8.0308  &&$     1.1784e-07$    &    3.9950\\
  80  &  &   $   3.7213e-11$ &   6.0040  & &$   3.6637e-14$    &  8.0099    &&$   5.0182e-14$    &     8.0145  &&$     7.3717e-09$    &    3.9987\\
  160 &  &   $   5.8065e-13$ &   6.0020  & &$   4.4409e-16$    &  -         &&$   6.6613e-16$    &     -       &&$     4.6083e-10$    &    3.9997\\
  320 &  &   $   8.6597e-15$ &   6.0672  & &$   4.4409e-16$    &  -         &&$   4.4409e-16$    &     -       &&$     2.8804e-11$    &    3.9999\\
\hline
\end{tabular}
    \caption{Errors and numerical orders calculating the integral at [0,1] of $f_1(x)=e^{x^2}$.}
    \label{tablaerrores}
\end{center}
\end{table}
\vspace{-1cm}
\subsection{Analysis of the number of evaluations}
As second example we approximate the integral of Runge function $f_2(x)=(1+25x^2)^{-1}$ at the interval $[-1,1]$
and show our results in Table \ref{tablaerrores2}. In order to compare with the classical rules we perform the same example using the Simpson rule but using a bigger number of evaluations. We will use, in this case, the notation $e^\Xi_M$, where $\Xi=p,S$ and $M$ is the number of evaluations. We can observe in Table \ref{tablaerrores2} that for this type of function where the interpolation process is not satisfactory, B-spline integration formula achieves the same order but lower error.
 \begin{table}[H]
\begin{center}
\begin{tabular}{ccccc}
  \hline
$M$   &      $e_M^1$       &   $e_M^2$     &         $e_M^3$  &         $e_M^S$   \\
 \hline
 15   &   $    1.8614e-03  $  &$   2.4084e-03    $    &$    2.4369e-03$    &$  5.3393e-03$     \\
 25   &   $    1.1867e-04  $  &$    7.6903e-06    $   &$    9.1477e-06$    &$  2.2269e-04$    \\
 45   &   $    3.0805e-05  $  &$    2.0297e-07    $   &$    2.8981e-07$    &$  4.5289e-07$    \\
 85   &   $    7.7038e-06  $  &$    1.2627e-08    $   &$    1.7991e-08$    &$  2.8097e-09$    \\
\hline
\end{tabular}
   \caption{Errors calculating the integral at [-1,1] of $f_2(x)=\frac{1}{1+25x^2}$ using different methods.}
    \label{tablaerrores2}
\end{center}
\end{table}
\vspace{-0.75cm}
In conclusion, we have developed some new numerical integration formulas for any order of approximation $p$. In some cases, the error is reduced faster with a lower number of evaluations in comparison to using the classical algorithms. Moreover, the flexibility and the simplicity of the new formulas allow us to make new high order methods with a lower cost. As future work, we suggest using a strategy based on hierarchical bases proposed by Speleers in \cite{speleers} in order to reduce the number of evaluations.
\vspace{-0.5cm}
{\small
\bibliographystyle{these}
}

\end{document}